\documentclass[11pt]{amsart}
\usepackage{amsmath,amssymb,latexsym}
\usepackage[small]{caption}
\usepackage{graphicx,color,mathrsfs,tikz}%wasysym,overpic,tikz,
\usepackage{subfigure,color}

\usepackage{cite}
\usepackage[colorlinks=true,urlcolor=blue,
citecolor=red,linkcolor=blue,linktocpage,pdfpagelabels,
bookmarksnumbered,bookmarksopen]{hyperref}
\usepackage[italian,english]{babel}
\usepackage{units}
\usepackage{enumitem}
\usepackage[left=2.6cm,right=2.6cm,top=2.75cm,bottom=2.75cm]{geometry}
\usepackage[hyperpageref]{backref}

\usepackage[colorinlistoftodos,prependcaption]{todonotes}
\numberwithin{equation}{section}
\newtheorem{theorem}{Theorem}[section]
\newtheorem{proposition}[theorem]{Proposition}
\newtheorem{lemma}[theorem]{Lemma}
\newtheorem{remark}[theorem]{Remark}
\newtheorem{open}[theorem]{Open problem}

\newtheorem{definition}[theorem]{Definition}
\theoremstyle{definition}

\renewcommand{\epsilon}{\eps}

\newcommand{\Ch}{{\rm Ch}}

\newcommand{\N}{{\mathbb N}}
\newcommand{\R}{{\mathbb R}}

\newcommand{\eps}{\varepsilon}

\newcommand{\pnorm}[2][]{\if #1'' \left|#2\right|_p \else \left|#2\right|_{#1} \fi}

\renewcommand{\theta}{\vartheta}

\title{New characterizations of Sobolev metric spaces}

\author[S. Di Marino]{Simone Di Marino}
\author[M.\ Squassina]{Marco Squassina}

\address[S. Di Marino]{Indam \newline\indent 
	Scuola Normale Superiore \newline\indent
	Piazza dei Cavalieri 7, 56126 Pisa, Italy}
\email{simone.dimarino@altamatematica.it}

\address[M.\ Squassina]{Dipartimento di Matematica e Fisica \newline\indent
	Universit\`a Cattolica del Sacro Cuore \newline\indent
	Via dei Musei 41, I-25121 Brescia, Italy}
\email{marco.squassina@unicatt.it}

\thanks{The authors are  members 
	of {\em Gruppo Nazionale per l'Analisi Ma\-te\-ma\-ti\-ca, la Probabilit\`a e le loro Applicazioni} (GNAMPA) of the {\em Istituto Nazionale di Alta Matematica} (INdAM)}



\begin{document}

\begin{abstract}
	We provide new characterizations of Sobolev ad BV spaces in doubling and Poincar\'e metric spaces in the
	spirit of the Bourgain-Brezis-Mironescu and Nguyen limit formulas holding 
	in domains of $\R^N$.
\end{abstract}
\maketitle

\section{Introduction}
\subsection{Overview}
Around 2001, J.\ Bourgain, H.\ Brezis and P.\ Mironescu, investigated 
\cite{bourg,bourg2,bre}  the asymptotic behaviour of a class on nonlocal functionals on a domain $\Omega\subset\R^N$, including those related to the norms 
of the fractional Sobolev space $W^{s,p}(\Omega),$ as $s\nearrow 1$. More precisely,  if $p\geq 1$ and $u\in W^{1,p}(\Omega)$, then 
\begin{equation*}
	\lim_{s \nearrow 1} (1-s)\int_\Omega\int_{\Omega} \frac{|u (x) -  u(y)|^p}{|x - y|^{N+ps}} \, d{\mathscr L}^N(x) \, d{\mathscr L}^N(y) = K_{p,N} \int_{\Omega} |\nabla u|^p \,d{\mathscr L}^N(x),
\end{equation*}
where $|\cdot|$ denotes the Euclidean norm, ${\mathscr L}^N$ the Lebesgue measure on $\R^N$ and
\begin{equation*}
	K_{p,N}=\frac{1}{p}\int_{{\mathbb S}^{N-1}} |{\boldsymbol \omega}\cdot x|^pd{\mathscr H}^{N-1} ,
\end{equation*}
being ${\boldsymbol \omega}\in {\mathbb S}^{N-1}$ arbitrary. By replacing
the Euclidean distance $|x - y|$ with a distance $d_K(x,y)=\|x - y\|_K$, where $K$ denotes the unit ball for $\|\cdot\|_K$,
it was recently proved \cite{anisotropic} that, if $u\in W^{1,p}(\Omega)$, then
\begin{equation*}
\lim_{s \nearrow 1} (1-s)\int_{\Omega}\int_{\Omega} 
\frac{|u (x) -  u(y)|^p}{d_K(x,y)^{N+ps}} \, d{\mathscr L}^N(x) \, d{\mathscr L}^N(y)= \int_{\Omega} \|\nabla u\|^p_{Z^*_pK}\, d{\mathscr L}^N(x),
\end{equation*}
where we have set 
\begin{equation*}
%\label{Lpmoment}
%\|u\|_{W^{1,p}_K}:=\left(\int_{\R^N}\|\nabla u\|^p_{Z^*_pK}dx\right)^{1/p},\qquad
\|\xi\|_{Z^*_pK}:=\left(\frac{N+p}{p}\int_K |\xi\cdot x|^p_p \, d{\mathscr L}^N(x)\right)^{1/p},\quad\,\, \xi\in\R^N.
\end{equation*}
Similar results hold for BV spaces \cite{davila,ponce} and for magnetic Sobolev spaces \cite{pinaioevecchi} and criteria for recognizing constants
among measurable functions can be obtained \cite{bre}.
The nonlocal norms thus converge in the limit as $s \nearrow 1$ to a Dirichlet type energy which depends on $p,N$ and on 
the distance $d_K$. 
More in general, it is natural to wonder if similar characterizations may hold for some classes of BV and Sobolev
spaces on a metric measure space $(X,d,\mu)$ in place of $\R^N$, at least in the case where some structural assumption is assumed on 
the measure $\mu$ acting on $X$. The general definition of Sobolev and BV space will be given in Section \ref{sec:sob}, and for every Sobolev function $u$ it is defined a weak upper gradient $|\nabla u|_w \in L^p(X)$, along with its \textit{Cheeger energy} (introduced by Cheeger in \cite{Ch99}) 
$$
\Ch_p (u):= \int_X | \nabla u|_w^p \, d\mu. 
$$
In particular  $\Ch_p$ will be l.s.c. with respect to the strong convergence in $L^p$ and so is a good generalization of the Dirichlet energy in an Euclidean context, where the two notion coincide. Moreover $W^{1,p}(X,d,\mu)$ is a Banach space with the norm $\|u\|_{1,p} = (\| u \|_{L^p}^p + \Ch_p(u))^{1/p}$.

We recall here (using  \cite{KMac, KZ}), that  whenever $\mu$ is doubling and it satisfies a $(1,p)$-Poincar\'e inquality, $W^{1,p}(X,\mu,d)$ coincides with the {\em Hajlasz-Sobolev} space, that is the space of $u\in L^p(X,\mu)$ 
such that there exists $g\in L^p(X,\mu)$ with
\begin{equation}
\label{difcond}
|u(x)-u(y)|\leq d(x,y)(g(x)+g(y)),\quad\text{for $\mu$ a.e. $x,y\in X$.}
\end{equation}
%As known, in this case $W^{1,p}(X,\mu,d)$ is a Banach space when endowed with the norm
%$$
%\|u\|_{W^{1,p}(X)}=\|u\|_{L^p(X)}+\inf\big\{\|g\|_{L^p(X)}: \text{$g$ satisfies \eqref{difcond}}\big\}.
%$$ 
Moreover we can choose $g$ such that  $\|g\|_p^p \leq C\cdot \Ch (u)$.

Furthermore if $p>1$ and $\Omega\subset\R^N$ is an extension domain, then $W^{1,p}(\Omega,{\mathscr L}^N,d)$
coincides with the usual space $W^{1,p}(\Omega)$ and the norms are equivalent.

For any $p\geq 1$ and $0<s<1$, the fractional space ${\mathcal H}^{s,p}(X,\mu,d)$ can be defined  
as the space of $u\in L^p(X,\mu)$ such that the Gagliardo seminorm $[u]_{{\mathcal H}^{s,p}(X)}$ is finite, where
$$
[u]_{{\mathcal H}^{s,p}(X)}:=\left(\int_X \int_X \frac{ |u(x) - u(y)|^p}{d(x,y)^{ps}\rho(x,y)}  \, d \mu(x) \, d \mu(y)\right)^{1/p},
$$
and $\rho$ is a doubling kernel for $\mu$ (see Definition \ref{def:doubker}).
%As a consequence of Theorem \ref{main}, we have a characterizations of functions in $W^{1,p}(X,\mu,d)$, namely $u\in W^{1,p}(X,\mu,d)$ if and only if 
%\begin{equation*}
%%\label{condizz}
%\limsup_{ s \nearrow 1}  (1-s)[u]_{W^{s,p}(X)}^p<\infty.
%\end{equation*}
A {\em fractional counterpart} of the Hajlasz-Sobolev spaces can also be introduced as follows. 
For $0<s<1$ we define $W^{s,p}(X,\mu,d)$ as the spaces of $u\in L^p(X,\mu)$ such that
there is a function $g\in L^p(X,\mu)$ with
$$
|u(x)-u(y)|\leq d^s(x,y)(g(x)+g(y)),
$$
for almost any $x,y\in X$. When the measure is $N$-Ahlfors it follows (see \cite{compara}) that
$$
{\mathcal H}^{s,p}(X,\mu,d)\hookrightarrow W^{s,p}(X,\mu,d) \hookrightarrow {\mathcal H}^{s-\eps,p}(X,\mu,d),
$$
for all $\eps\in (0,s)$, so that the two spaces are comparable.

\vskip3pt
\noindent
The main goal of this paper is to provide a proof of the connection between 
$$
\limsup_{ s \nearrow 1}  (1-s) \int_X \int_X \frac{ |u(x) - u(y)|^p}{d(x,y)^{ps}\rho(x,y)}  \, d \mu(x) \, d \mu(y)<+\infty 
$$
and $u\in W^{1,p}(X)$ for $p>1$ or $u\in BV(X)$ for $p=1$.  
%It is well know that Riemaniann manifold with Ricci curvature bounded from below are PI spaces.
A second characterization we want to provide is in terms of the family nonlocal integrals
$$
u\mapsto \iint_{\{|u(x)-u(y)|>\delta\}} \frac{\delta^p}{d(x,y)^{p}\rho(x,y)}  \, d \mu(x) \, d \mu(y) .
$$
In the Euclidean case $X=\R^N$, Nguyen \cite{ngu1,NgSob2,Ng11,NgGamma} 
(see also the recent works \cite{bre-linc,BHN3,BHN2,BHN-0,BrNgu} by Brezis and Nguyen)
proved that, if $p>1,$ then $u\in W^{1,p}(\R^N)$ if and only if $u\in L^p(\R^N)$ and 
$$
\sup_{0<\delta<1}\iint_{\{|u(x)-u(y)|>\delta\}} \frac{\delta^p}{|x-y|^{N+p}}  \, d{\mathscr L}^N(x) \, d {\mathscr L}^N(y)<+\infty,
$$
in which case
$$
\lim_{\delta\searrow 0}\iint_{\{|u(x)-u(y)|>\delta\}} \frac{\delta^p}{|x-y|^{N+p}}  \, d{\mathscr L}^N(x) \, d {\mathscr L}^N(y)
=K_{p,N}\int_{\R^N} |\nabla u|^p \,d{\mathscr L}^N(x).
$$
In the case $p=1$ this property fails, in general \cite{BrNgu}.

\subsection{Main results}
In the following $(X,d,\mu)$ denotes a metric measure space with measure $\mu$.

\begin{definition}[Doubling]\rm 
We say that $\mu$ is a doubling measure  if there exists a constant $c_D$ such that 
$$
 \mu (B(x,2r))\leq c_D \mu(B(x,r)),\quad \text{for all $x\in {\rm supp}(\mu)$ and any $r>0$.}
$$
\end{definition}

\begin{definition}[Doubling kernel]\label{def:doubker} \rm 
Let $(X,d, \mu)$ be a metric space with $\mu$ doubling. We say $\rho : X \times X \to \mathbb{R}$ is a \emph{doubling kernel} if there exists a constant $C_{\rho}>0$ such that 
$$
 \frac 1{C_{\rho}} \mu (B(x,d(x,y)))\leq \rho(x,y) \leq C_{\rho} \mu(B(x,d(x,y)) ,\quad \text{for all $x, y \in {\rm supp}(\mu)$.}
$$
\end{definition}

There are several examples of doubling kernels used in the literature: here we list a few, denoting with $\rho_1(x,y)=\mu(B(x,d(x,y))$ and $\rho_2(x,y)=\mu(B(y, d(x,y))$

$$ \rho_1, \quad \rho_2, \quad \rho_1+\rho_2, \quad \frac{ \rho_1 + \rho_2 }{\rho_1 \rho_2} , \quad \sqrt{ \rho_1 \rho_2 } , $$
and in general $f(\rho_1, \rho_2)$ where $\min \{t,s\} \leq f(t,s) \leq \max \{t,s\}$. In the special case when $\mu$ is $N$-ahlfors, also $d(x,y)^N$ is a doubling kernel.

\begin{definition}[Poincar\'e inequality]\rm 
	We say that $\mu$ satisfies a $(1,p)$-Poincar\'e inequality if there is $c_P>0$ such that for any ball $B\subset X$ of radius $t>0$
	\begin{align*}
	\int_B | u_B - u(x)|^p \, d \mu(x) &\leq t^p c_P \int_{B} g^p(x) \, d \mu(x), \quad u\in W^{1,p}(X), \,\, \text{(Sobolev case),}\\
	\int_B | u_B - u(x)| \, d \mu(x) &\leq t c_P |Du|(B), \quad u\in BV(X),\,\, \text{(BV case)}.
	\end{align*}
\end{definition}

Notice that this definition is a bit different and less general than the usual one, that allows the integral on the right hand side to be performed over a larger ball $B(x,\tau r)$, for some $\tau \geq 1$. We prefer to stick to this version since the proof becomes clearer, but of course modifications can be done in order to fit the more general definition.

\vskip5pt
The main results of the paper are the following.

\begin{theorem}[BBM type characterization]
	\label{main}
Let $p\geq 1$. Assume that $(X,d,\mu)$ is a metric measure space and $\mu$ is doubling and satisfies a $(1,p)$-Poincar\'e inequality. Let $\rho$ be a doubling kernel: 
then there exist $C_U>0$ and $C_L>0$ depending on $p,N,C_{\rho},c_P,c_D$ such that for every $u \in L^p(X)$ we have:
\begin{align*}
& \limsup_{ s \nearrow 1}  (1-s) \int_X \int_X \frac{ |u(x) - u(y)|^p}{\rho(x,y)d(x,y)^{ps}}  \, d \mu(x) \, d \mu(y)  \leq C_U \Ch_p(X), \\
& \liminf_{ s \nearrow 1}  (1-s) \int_X \int_X \frac{ |u(x) - u(y)|^p}{\rho(x,y)d(x,y)^{ps}}  \, d \mu(x) \, d \mu(y)  \geq C_L \Ch_p(X).
\end{align*}
%where $\Ch_p(E) = \int_E g^p d\mu$ in the case $u\in W^{1,p}(X)$ and $\nu(E)= |Du|(E)$ 
%in the case $u\in BV(X)$.
\end{theorem}
\vskip3pt

In the case $p>1$, Theorem~\ref{main} was already obtained 
in \cite{munnier} with a different and more involved technique, while the BV case, 
to the best of our knowledge, was open. The details in \cite{munnier} are present only 
for Ahlfors measures, in which case an upper bound is firstly obtained on balls by exploiting the definition \eqref{difcond} and 
$$
\sup_{0<s<1}(1-s)\int_{B(y,r)}\frac{1}{d(x,y)^{N-p(1-s)}}d\mu(x)<+\infty,
\,\quad\text{for $\mu$ a.e. $y\in X$ and all $r>0$,}
$$
which essentially follows from the fact that the measure of the balls of radius $t$ grows $N$-polynomially. 
On the contrary, the lower bound in \cite{munnier} is extremely involved and based, among other tools, upon some deep differentiation
result contained in \cite{cheeger}, which says that every Lipschitz map from $X$ into a Banach space with
the Radon-Nikodym Property is almost everywhere differentiable. 
%\footnote{Add here further comments here about the comparison 
%	between our proof (in the lower bound) and the one by Murrier.}

\vskip3pt
The following result is instead new in metric spaces, up to our knowledge.

\begin{theorem}[Nguyen type characterization]
	\label{main2}
Let $p>1$. Assume that $(X,d,\mu)$ is a metric measure space and $\mu$ is doubling and satisfies a $(1,p)$-Poincar\'e inequality. Let $\rho$ be a doubling kernel: then there exist $C_U>0$ and $C_L>0$ depending on $p,N,C_{\rho},c_P,c_D$ such that for every $u \in L^p(X)$ we have:
$$	C_L \Ch_p(u) \leq \limsup_{ \delta \searrow 0}  \mathop{\int_X\int_X}_{\{|u(x)-u(y)|>\delta\}} \frac{\delta^p}{\rho(x,y)d(x,y)^{p}}  \, d \mu(x) \, d \mu(y)  \leq C_U \Ch_p(u). $$%	\footnote{In order to get this result, I suggest to
%		take a look at \cite[Lemma 2]{ngu1}. I also think that the Maximal Function Theorem for doubling
%		metric measure spaces, \cite[Theorem 2.2]{heinonen} will be needed.} 
\end{theorem}

We will now outline the proof of the results. 

In the case of the BBM type characterization the key tool is a clever use of Fubini theorem that let us compare the quantity we want to estimate with 
$$ {\mathcal S}_t := \int_X \frac 1{ \mu(B(x',t))^2}\iint_{B(x',t) \times B(x',t)} |u(x)-u(y)|^p \, d \mu(x) \, d \mu (y) \, d \mu (x'),$$
which is very reminiscent of the Korevaar and Shoen  definition of Sobolev functions \cite{KS}. As a result of this estimate we show that, in order to conclude,  it is sufficient to have a good bound on the liminf/limsup of $\frac { {\mathcal S }_t}{t^p}$ as $t \to 0$. Then an easy application of the Poincar\'e inequality will give us the upper bound while for the lower bound we use Lemma \ref{lem:delta0}, and the fact that $\frac { {\mathcal S }_t}{t^p}$ can be seen as the energy of $g_{t}$, which, up to a constant, is an upper gradient up to scale $t/2$ of the function $u^t$, that in turn is an approximation of $u$.

As for the Nguyen-type characterization, for the upper bound we use the Hajlasz-Sobolev characterization of Sobolev functions, while for the lower bound we again use cleverly Fubini (as done by Nguyen in its original work \cite{ngu1}), and then we use again Lemma \ref{lem:delta0}, but this time the proof is more involved because the estimate is not so direct.

\begin{remark}\rm
	Concerning the case $p=1$ in the previous Theorem \ref{main2}, in general, already in the Euclidean case, the assertion cannot
	hold true, in the sense that examples can be found \cite{p-uno,BrNgu} of functions $u$ in $W^{1,1}(\Omega)$ such that  
	$$
	\lim_{ \delta \searrow 0}   \mathop{\int_\Omega\int_\Omega}_{\{|u(x)-u(y)|>\delta\}} \frac{\delta}{|x-y|^{N+1}}  \, d {\mathcal L}^N(x) \, d {\mathcal L}^N(y)=+\infty.
	$$
	Moreover it is desirable to have a lower bound of the $\liminf$ as in Theorem \ref{main}, but this is more difficult and in the euclidean context it was solved in \cite{p-uno}.
	%On the other hand, it is possible to prove the inequality
	%$$
	% \liminf_{ \delta \searrow 0}   \mathop{\int_X\int_X}_{\{|u(x)-u(y)|>\delta\}} \frac{\delta}{d(x,y)^{N+1}}  \, d \mu(x) \, d \mu(y)   \geq C_L \int_X g d\mu,
	% $$
	% by arguing as in the proof of Theorem \ref{main2}.\footnote{Add details.}
\end{remark} 

\begin{open}\rm
	Let $u\in L^1(X)$. Let $\{\delta_n\}_{n\in\N}\subset\R^+$ with $\delta_n\to 0$. Assume that
	$$
	\liminf_{ \delta_n \to 0}   \mathop{\int_X\int_X}_{\{|u(x)-u(y)|>\delta_n\}} \frac{\delta_n}{\rho(x,y) d (x,y)}  \, d \mu(x) \, d \mu(y)<+\infty. 
	$$
	Then $u\in BV(X)$ and there exists a positive constant $C$ such that
	$$
	\liminf_{ \delta_n \to 0}   \mathop{\int_X\int_X}_{\{|u(x)-u(y)|>\delta_n\}} \frac{\delta_n}{\rho(x,y) d(x,y)}  \, d \mu(x) \, d \mu(y)   \geq C \Ch_1(X)
	$$
	This rather subtle assertion was proved in the Euclidean case in \cite{p-uno} (see also \cite{BrNgu}).
\end{open} 

\section{Preliminaries}

In this section we will introduce the well established theory of Sobolev spaces in metric measure spaces, as well as some technical results that will be needed in the proofs.

\subsection{Sobolev spaces in metric measure spaces}\label{sec:sob}

Several equivalent definition of $W^{1,p}(X,\mu,d)$ and $BV$ are available in the literature: we refer to \cite{AGS, ADM, DMPhD, BB, S, libroS} as general references. We will use the definition of Sobolev spaces given in \cite{AGS} (and in \cite{ADM} for BV spaces), where it is also proved to be equivalent to the more common definition of newtonian spaces $N^{1,p}$, defined for example in \cite{S}. In the sequel $p$ will be the Sobolev exponent and $q$ is its dual exponent, namely $1/p +1/q=1$.

We will denote by $AC([0,1];X)$ the space of absolutely continuous curves $\gamma: [0,1] \to X$, for which it is defined the metric derivative $|\gamma'|$ almost everywhere. Moreover we set $e_t :AC([0,1];X) \to X$ as the evaluation of $\gamma$ at time $t$, namely $e_t(\gamma)=\gamma(t)$. Another useful definition is that of push forward: given a Borel function $f:X \to Y$ and a measure $\mu$ on $X$ we define  $\nu = f_{\sharp} \mu$ as the measure on $Y$ such that $\nu(A)=\mu (f^{-1}(A))$.    

A key useful concept for Sobolev Spaces is the upper gradient.
%
%\begin{definition} A continuous curve $\gamma: [0,1] \to X$ is said to belong to $AC^p([0,1];X)$ if there exists $g \in L^p(0,1)$ such that
%\begin{equation}\label{eqn:accurve} d(\gamma(s), \gamma(t)) \leq \int_s^t g(r) \, dr \qquad \forall 0 \leq s < t \leq 1.\end{equation}
%For such a curve there exists the metric derivative $| \gamma|' (t) := \lim_{s \to t} \frac {d(\gamma(s),\gamma(t))}{|s-t|}$, which is also the least function $g$ that satisfies \eqref{eqn:accurve}. We call the length of $\gamma$ the quantity $\ell(\gamma) = \int_0^1 | \gamma|'(t)\, dt$.
%\end{definition}
%
%The curves in $AC^1([0,1];X)$ (where we will often drop the exponent), are called absolutely continuous curves.

\begin{definition}[Upper gradients] 
Let $f:X \to \mathbb{R}$ and $g:X \to [0, \infty]$. We say that $g$ is an \emph{upper gradient} for $f$ if for every curve $\gamma \in AC([0,1];X)$ we have the so called upper gradient inequality
\begin{equation}\label{eqn:ugineq} | f( \gamma(1))- f(\gamma(0))| \leq \int_0^1 g(\gamma(t))| \gamma|'(t) \, d t. \end{equation}
We will often substitute the right hand side with the shorter notation $\int_{\gamma} g$. Moreover we say that $g$ is an upper gradient of $f$ up to  scale $\delta$ if  \eqref{eqn:ugineq} is satisfied for every $\gamma$ such that $\ell(\gamma) > \delta$.
\end{definition}

We will need one more class object in order to define the Sobolev Spaces: the $p$-plans.
%The idea behind the next definition is that while for a Lipschitz function $f$ we have that $|\nabla f|$ is an upper gradient, if $f$ is merely Sobolev then the upper gradient inequality is not true for every curve, but for \emph{almost every curve}: we need thus to introduce a concept that will allow us to talk about negligible sets of curves.

\begin{definition}[$p$-plans] 
	Let $\pi$ be a probability measure on $C([0,1];X)$. We say $\pi$ is a $p$-plan if
\begin{itemize}
%\item $\pi$ is concentrated on $AC^q([0,1];X)$;
\item there exists $C>0$ such that $(e_t)_{\sharp}\pi \leq C \mu$ for every $0 \leq t \leq 1$;
\item there exists $b_{\pi} \in L^q(X,\mu)$,  called barycenter of $\pi$, such that
$$ \int_{AC} \left( \int_{\gamma} g \right) \, d \pi = \int_X g \cdot b_{\pi} \, d \mu, \qquad \forall g \in C_b(X,d).$$
\end{itemize}
We will say that a property on $AC$ is true for $p$-almost every curve if it is true for $\pi$-almost every curve, for every $p$-plan $\pi$. Conversely a set of curves $\Gamma$ is said to be $p$-null or $p$-negligible if $\pi(\Gamma)=0$ for every $p$-plan $\pi$.
\end{definition}

With this notion of $p$-almost every curve, we can relax the notion of upper gradient, and with this relaxed notion we can define the Sobolev Space.

\begin{definition}[$p$-weak upper gradient]\label{def:wug} A function $g \in L^p(X,\mu)$ is a $p$-weak upper gradient for $f \in L^p(X, \mu)$ if for $p$-almost every curve $\gamma$ we have that $f \circ \gamma$ is $W^{1,1}$ and moreover 
	$$
	\left|\frac d{dt} f \circ \gamma (t) \right| \leq g ( \gamma(t)) | \gamma|' (t),
	$$ 
	for almost every $t \in [0,1].$
\end{definition}

\begin{definition}[Sobolev space] \label{def:sob} Let $p\geq 1$. A function $f \in L^p(X, \mu)$ belongs to $W^{1,p}(X,d , \mu)$  if equivalently
\begin{itemize}
\item[(a)]  $f$ has a $p$-weak upper gradient; then there exists a minimal weak upper gradient  (in the $\mu$-a.e. sense), denoted by $|\nabla f|_w$. 
\item[(b)] (only if $p>1$) there exists a constant $C$ such that for every $p$-plan $\pi$ we have
$$ \int_{AC} | f(\gamma(0)) - f(\gamma(1))| \, d \pi \leq \| b_{\pi}\|_q \cdot C^{1/p}.$$
\item[(c)] there exists $ g \in L^p(X, \mu)$ such that for every $p$-plan $\pi$ we have
$$ \int_{AC} | f(\gamma(0)) - f(\gamma(1))| \, d \pi \leq \int_X g \cdot b_{\pi} \, d \mu.$$
\end{itemize}
Moreover the least constant $C$ in (b) is equal to $\int_X | \nabla f|_w^p \, d \mu$ and the minimal $g$ that satisfies $(c)$ is again $| \nabla f |_w$.
\end{definition} 

\begin{definition}[BV space] \label{def:bv} A function $f \in L^1(X, \mu)$ belongs to $BV(X,d , \mu)$ if equivalently
\begin{itemize}
\item[(a)] $f\circ \gamma$ is $BV$ for $p$-almost every curve and there exists a measure $\nu$ such that for every $p$-plan $\pi$ we have
$$\int_{AC} \gamma_{\sharp} |D (f \circ \gamma) | ( A)  \, d \pi \leq \| b_{\pi}\|_{\infty} \cdot \nu(A) \qquad \forall A \subseteq X \text{ open set}.$$
\item[(b)] there exists a constant $C$ such that for every $1$-plan $\pi$ we have
$$ \int_{AC} | f(\gamma(0)) - f(\gamma(1))| \, d \pi \leq \| b_{\pi} \|_{\infty} \cdot C.$$
\item[(c)] there exists a finite measure  $\nu$ such that for every $1$-plan $\pi$ we have
$$ \int_{AC} | f(\gamma(0)) - f(\gamma(1))| \, d \pi \leq \int_X b^*_{\pi} \, d \nu,$$
where $b^*_{\pi}$ denotes the upper semicontinuous relaxation of $b_{\pi}$.
\end{itemize}
Moreover the minimal $\nu$ in either (a) or (c) is denoted by $|Df|$ and the least constant $C$ in (b) is equal to $|Df|(X)$.
\end{definition} 

In the following we will denote 
\begin{equation}
\label{charge}
\Ch_p ( f ):=  \int_X |\nabla f|_w^p \, d \mu,\quad \text{for $p>1$},\qquad 
\Ch_1(f):=|Df|(X).
\end{equation}
For the next lemma in the case $p>1$ we refer the reader to\cite{Ambr1}.

\begin{lemma}[Semicontinuity]
	\label{lem:delta0} 
	Let $p \geq 1$ and let $f_n,g_n \in L^p_{{\rm loc}}(X,\mu)$ be functions such that $g_n$ is an upper gradient up to scale $\delta_n$ of $f_n$. Suppose that $\delta_n \downarrow 0$, $f_n \to f$ in $L^p(X,\mu)$ and $g_n \rightharpoonup g$ weakly in $L^p_{{\rm loc}}(X,\mu)$ (respectively in the sense of measure). Then $g$ is a $p$-weak upper 
	gradient for $f$ (respectively we have $|Df| \leq g$). In particular we have also
$$   
\liminf_{n \to \infty} \int_{X} g_n^p \, d \mu \geq \sup_{R >0} \liminf_{n \to \infty} \int_{B(x_0, R)} g_n^p \, d \mu \geq \Ch_p(f)
$$
\end{lemma}
\begin{proof}
For every $M >0$, let us denote by $\mathcal{A}_M \subseteq AC([0,1];X)$ the set
 $$\mathcal{A}_M:=\{ \ell(\gamma) \geq \tfrac 1M \} \cap \{ \gamma([0,1]) \subseteq B_M(x_0)\}.$$
If $\ell(\gamma)>0$ we have that $\gamma \in \mathcal{A}_M$ for $M= \sup \left\{ \ell(\gamma)^{-1} , d(\gamma(0),x)+\ell(\gamma)\right\}$, so, in particular 
$$ 
AC([0,1];X)= \{ \ell(\gamma)=0\} \cup \bigcup_{n=1}^{\infty} \mathcal{A}_n. 
$$
We can now define $\mathcal{B}_n=\{ \ell(\gamma)=0\} \cup  \mathcal{A}_n$. Let us consider $\pi_n : = \pi|_{\mathcal{B}_n}$ and compute
$$ 
\int_{AC} | f(\gamma_1) -f(\gamma_0)| \,d \pi_n \leq \int_{AC} |f_m(\gamma_0) - f_m(\gamma_1)| \, d \pi_n + 2C \int_{B(x_0,n)} | f_m -f|\, d \mu, 
$$
where we used the triangular inequality and the first property of $p$-plans. Then we take $m$ big enough such that $\delta_m \leq \tfrac 1n$ and in this way we can use the upper gradient property $\pi_n$-almost everywhere (notice also that if $\ell(\gamma)=0$ the upper gradient property is trivial) to get

\begin{align*} \int_{AC}  | f(\gamma_1) -f(\gamma_0)| \,d \pi_n &\leq \int_{AC}\left( \int_{\gamma} g_m \right) \, d \pi_n + 2C \int_{B(x_0,n)} | f_m -f|\, d \mu  \\ &\leq \int_{B(x_0,n)} g_m \cdot b_{\pi} \, d \mu + 2C \int_{B(x_0,n)} | f_m -f|\, d \mu
\end{align*}
Taking the limit as $m \to \infty$ (using $b_{\pi} \in L^q$ and the weak convergence of $g_m$ to $g$), and then taking $n \to \infty$ we get precisely Definition \ref{def:sob} (respectively \ref{def:bv}) (c), and so we can conclude.
\end{proof}

\subsection{Preliminaries on doubling spaces equipped with Poincar\'e inequality}

Let us define a regularization operator $M_t : L^p(X) \to L^p(X)$
$$ M_tf (x)= \frac 1{\mu(B(x,t))} \int_{B(x,t)} f(y) \, d \mu(y). $$
We state its main properties

\begin{lemma}[Boundedness of $M_t$]
	\label{lem:mtconv} 
	Let $\mu$ be a doubling measure with doubling constant $c_D$. Then $M_t$ is a linear bounded operator from $L^p(X)$ to itself, in particular
	$$
	\| M_t f\|_p \leq c_D \| f\|_p,\quad \text{for every $f \in L^p(X)$.}
	$$ 
	Moreover we have $\|M_t f -f\|_p \to 0$ as $t \to 0$ for every $f \in L^p(X)$.
\end{lemma}

\begin{proof} For the first part we use first Jensen inequality
$$ \int_X  \left (  \frac 1{\mu(B(x,t))} \int_{B(x,t)} f \, d \mu \right)^p \, d \mu \leq \int_X   \frac 1{\mu(B(x,t))} \int_{B(x,t)} f^p \, d \mu \, d \mu, $$
and then Fubini to obtain

$$ \int_X   \frac 1{\mu(B(x,t))} \int_{B(x,t)} f^p \, d \mu \, d \mu = \int_X f^p(y) g_t(y) \, d \mu, $$
where $g_t(y)= \int_{B(y,t)} \frac 1{\mu(B(x,t))} \, d \mu$. Using the doubling property we get
$$ g_t (y) \leq \int_{B(y,t)} \frac {c_D}{\mu(B(x,2t)} \, d \mu \leq \int_{B(y,t)} \frac {c_D}{\mu(B(y,t)} \, d \mu = c_D.$$
The convergence of $M_tf$ to $f$ is obvious for Lipschitz functions with bounded support and then we conclude using the boundedness of $M_t$ and the density of Lipschitz functions in $L^p(X)$.
\end{proof}

\begin{lemma}\label{lem:key} 
	If  $\mu$ is doubling, there exist $C>0$ such that for every $x \in X$, $r>0$, we have
$$ \int_{ \{d(x,y) \geq r\}} \frac 1{ \rho(x,y) d(x,y)^p } \, d \mu (y) \leq \frac C{r^p}. $$
\end{lemma}

\begin{proof} We consider the annuli $A_i(x)= \{ 2^ir \leq d(x,y) < 2^{i+1} r\}$. Now, whenever $y \in A_i$ we have $d(x,y) \geq 2^{i+1}r$, but also $\rho(x,y) \geq \tfrac 1C \mu(B(x,2^ir))$, since $\mu$ is doubling and $\rho(x,y)$ is comparable to $\mu(B(x,d(x,y)))$. We thus estimate
$$
\int_{\{ d(x,y) \geq r\}} \frac 1{ \rho(x,y) d(x,y)^p } \, d \mu (y)  = \sum_{i=0}^{\infty} \int_{A_i} \frac 1{ \rho(x,y) d(x,y)^p } \, d \mu (y) \leq \sum_{i=0}^{\infty} C \frac{ \mu( A_i )}{ \mu(B(x, 2^i r)) r^p 2^{ip+p}}.
$$
In the end we use $\mu(A_i) \leq \mu(B(x,2^{i+1}r))$ and then the doubling condition again to get
$$ 
\int_{ \{d(x,y) \geq r\}} \frac 1{ \rho(x,y) d(x,y)^p } \, d \mu (y)  \leq \frac{ C}{r^p} \frac 1{2^p-1} \leq  \frac{ C}{r^p},
$$
which concludes the proof.
\end{proof}

In the spirit of the Hajlasz-Sobolev space we then state the following

\begin{proposition}\label{prop:HK} Let $p>1$, $\mu$ be a doubling measure that satisfies a $(1,p)$-Poincar\'e inequality. Then for every $r>0$ there exists a constant $C_r$ such that for every $u \in W^{1,p}(X,d,\mu)$  there exists $g \in L^p$ such that $\|g \|_p^p \leq C_r \cdot \Ch_p (X)$ and
$$ |u(x)-u(y)| \leq d(x,y) (g(x) + g(y)) \qquad \forall x,y \in X, \; d(x,y) \leq r. $$
\end{proposition}
\begin{proof}
It is sufficient to combine the results from \cite{KMac} and \cite{KZ}, along with the boundedness of the maximal function operator in doubling spaces.
\end{proof}

\section{Proof of Theorem \ref{main}}

\noindent
We prove separately the upper and the lower bound.

\subsection{Upper bound of (doubling) Theorem \ref{main}}
For every ball $B=B(x',t)$, denoting by 
$$
u_B := \frac 1{\mu(B)} \int_B u \, d \mu,
$$ 
we have
\begin{align}\label{eqn:compamean}
\mu(B) \int_B | u_B - u(x)|^p \, d \mu(x) & \leq  \int_B \int_B | u(x) -u(y)|^p \, d \mu(x) \, d \mu (y) \\
& \leq 2^{p} \mu(B) \int_B  | u_B - u(x)|^p \, d \mu(x).  \notag
\end{align}
The first inequality follows by H\"older inequality, while the second one follows from 
the elementary inequality $|a+b|^p \leq 2^{p-1} |a|^p +|b|^p$ applied with $a=u(x)-u_B$ and $b=u_B-u(y)$. 
We now write 
$$
\frac 1{d(x,y)^{ps}}= ps \int_{d(x,y)}^{\infty} \frac 1 {t^{ps+1}} \, dt.
$$ 
Then we apply the Fubini-Tonelli Theorem and get in turn
\begin{align}\label{eqn:fubini1}
 \int_{X \times X} \frac {|u(x)-u(y)|^p}{\rho(x,y) d(x,y)^{ps}} \, d \mu(x) \, d \mu(y) & =ps \int_X \int_X \int_{d(x,y)}^{\infty} \frac{ |u(x) - u(y)|^p}{\rho(x,y) t^{ps+1}} \,dt \, d \mu(x) \, d \mu(y) \\
 &= ps \int_0^{\infty} \frac 1{t^{ps+1}} 
\iint_{\{d(x,y) \leq t\}} \frac{|u(x)-u(y)|^p}{\rho(x,y)} \,d\mu(x) \, d \mu(y). \notag
\end{align}
Now, let us define the quantities
\begin{align*}
{\mathcal K}_t &:=  \iint_{\{d(x,y) \leq t\}} \frac{|u(x)-u(y)|^p}{\rho(x,y)} \,d\mu(x) \, d \mu(y), \\
{\mathcal H}_t &:=  \iint_{\{d(x,y) \leq t\}} \frac{|u(x)-u(y)|^p}{\sqrt{ \mu(B(x,t)) \mu(B(y,t))}} \,d\mu(x) \, d \mu(y), \\
{\mathcal S}_t &:= \int \frac 1{ \mu(B(x',t))^2}\iint_{B(x',t) \times B(x',t)} |u(x)-u(y)|^p \, d \mu(x) \, d \mu (y) \, d \mu (x').
\end{align*}

We will prove a lemma that deals with relations between these quantities, and then an estimate from above of $\mathcal{S}_t$.

\begin{lemma}\label{lem:HKS} There exist $0< c < C < \infty$ depending only on the doubling constant such that for every $t>0$ we have
\begin{itemize}
\item[(i)] ${\mathcal H}_t \leq {\mathcal K}_t \leq C \sum_{k=0}^{\infty} \mathcal{H}_{ t/2^k}$;
\item[(ii)] $c {\mathcal H}_{t/2} \leq {\mathcal S}_t \leq C {\mathcal H}_{2t}$;
\item[(iii)] $\mathcal{S}_t \leq Ct^p \Ch_p(u)$.
\item[(iv)] if $t\geq 1$ then $\mathcal{K}_t \leq \mathcal{K}_1+ C\log_2 (2t)\int_X u^p \, d \mu$
\end{itemize}

\end{lemma}

%\begin{lemma}\label{lem:Subound} Let us consider $u \in W^{1,p}$ (respectively $u \in BV$ if $p=1$); then we have that 
%$$\mathcal{S}_t \leq \begin{cases} C (2t)^p \Ch_p (u) \quad & \text{ if }p>1 \\ Ct |Du|(X) & \text{ if }p=1. \end{cases}$$
%\end{lemma}
Before proving Lemma \ref{lem:HKS} we use it to deduce the upper bound: first of all we have 
\begin{equation}\label{eqn:stimaK} \mathcal{K}_t \leq \frac Cc \sum_{k=0}^{\infty}\mathcal{S}_{2t/2^k} \leq \frac {C^2}c \Ch_p(X) \sum_{k=0}^{\infty} \left(  \frac {4t}{2^k} \right)^p \leq C t^p \Ch_p(X).\end{equation}
Now we can use \eqref{eqn:fubini1} in order to find 
$$ (1-s) \int_{X \times X} \frac {|u(x)-u(y)|^p}{\rho(x,y) d(x,y)^{ps}} \, d \mu(x) \, d \mu(y) = (1-s)ps \int_0^{\infty} \frac {\mathcal{K}_t}{t^{ps+1}} dt.$$
Splitting the last integral in $t\leq1$ and $t>1$ will let us conclude using \eqref{eqn:stimaK} in the first part and Lemma \ref{lem:HKS} (iv) for the second part:

\begin{align*}
 (1-s)ps \int_0^{\infty} \frac{\mathcal{K}_t}{t^{ps+1}} \, dt & \leq (1-s)psC \int_0^1 \frac {\Ch_p(X)}{t^{p(s-1)+1}} \, dt + (1-s)ps \|u\|_p^p C \int_1^{\infty} \frac{\log_2(2t)}{t^{ps+1}} \, dt \\
 & \leq Cs \cdot \Ch_p(X) + (1-s) C \| u\|_p^p \int_0^{\infty} (1+\tfrac {\tau}{ps})e^{-\tau} \, d\tau \\
 & =  Cs \cdot \Ch_p(X) + (1-s) C \| u\|_p^p  (1+\tfrac 1{ps}).
 \end{align*}

In particular, letting $ s \to 1$ we obtain the upper bound.

\begin{proof}[Proof of Lemma \ref{lem:HKS}] Every constant inside this proof will depend on $c_D, c_P$, and possibly $p$.

\begin{itemize} 
\item[(i)] The inequality ${\mathcal H}_t \leq {\mathcal K}_t $ is trivial. The other inequality comes from the fact that 
$$ {\mathcal K}_t = \sum_{k=0}^{\infty} \iint_{\{\frac t{2^{k+1}} \leq d(x,y) \leq \frac t{2^k} \}} \frac{|u(x)-u(y)|^p}{\rho(x,y)} \,d\mu(x) \, d \mu(y); $$
then in every term we have 
$$
\rho(x,y) \geq \sqrt{ \mu(B(x, \tfrac t{2^{k+1}} )) \mu(B(y,  \tfrac t{2^{k+1}} ))} \geq \frac 1 {c_D} \sqrt{ \mu(B(x,  \tfrac t{2^{k}}) ) \mu(B(y, \tfrac t{2^{k}}) )}, 
$$
and so we have
$$  {\mathcal K}_t \leq c_D \sum_{k=0}^{\infty} \mathcal{H}_{t/2^k}.$$

\item[(ii)]%We want to prove that $c {\mathcal H}_{t/2} \leq {\mathcal S}_t \leq C {\mathcal H}_{2t},$ for some positive constants $c,C$. 
Let us begin by writing more explicitely ${\mathcal S}_t$, by doing the integration in $x'$ first, which yields
$$
{\mathcal S}_t = \int_X \int_X |u(x)-u(y)|^p f_t(x,y)\, d \mu (x) \, d \mu (y),
$$
where we have set
$$
f_t(x,y) := \int_{ B(x,t) \cap B(y,t)} \frac 1{ \mu (B(x',t))^2} \, d \mu(x').
$$
Thus, it is sufficient to prove that 
$$
\frac c{\sqrt{ \mu(B(x,t)) \mu(B(y,t))}} \chi_{\{d(x,y) \leq t/2\}} \leq f_t(x,y) \leq \frac C{\sqrt{ \mu(B(x,2t)) \mu(B(y,2t))}} \chi_{\{d(x,y) \leq 2t\}}.
$$ 
For the second inequality, if $d(x,y) > 2t $ we have $f_t(x,y)=0$, since $B(x,t) \cap B(y,t) = \emptyset$. 
Moreover we can bound from above using $ \mu(B(x,t)) \leq \mu(B(x',2t)) \leq c_D \mu (B(x',t))$ and the same for $y$:
$$f_t(x,y) \leq c_D^2 \frac{ \mu( B(x,t) \cap B(y,t))}{\sqrt{ \mu(B(x,t)) \mu(B(y,t))}\mu(B(x,t))} \leq \frac {c_D^3}{\sqrt{ \mu(B(x,2t)) \mu(B(y,2t))}}. $$ 
For the first inequality we need only to check that if $d(x,y) \leq t/2$ then $f_t$ is bounded from below. But in this case we have $B(x',t) \subseteq B(x,2t) $ and so $\mu(B(x',t)) \leq \mu(B(x,2t))  $ and the same is true for $y$. In particular, since this time $B(x,t/2) \subset B(x,t) \cap B(y,t)$, we get 
$$f_t(x,y) \geq \frac{\mu(B(x,t/2)) }{\sqrt{ \mu(B(x,2t) \mu (B(y, 2t)} B(x, 2t)} \geq \frac 1{c_D^4\sqrt{ \mu(B(x,t)) \mu(B(y,t))}}. $$

\item[(iii)] We use the Poincar\'e inequality in the form (remember that $B$ is a ball of radius $t$)
\begin{align*}
\int_B | u_B - u(x)|^p \, d \mu(x) &\leq t^p c_P \int_{B} g^p(x) \, d \mu(x),  \\
\int_B | u_B - u(x)| \, d \mu(x) &\leq t c_P |Du|(B).
\end{align*}
In the spirit of treating the Sobolev and the BV case together,
we can write 
$$
\nu(E) = \int_E g^p d\mu,\qquad \nu(E)= |Du|(E),
$$ 
respectively. We then have, using Equation \eqref{eqn:compamean} and Poincar\'e inequality
\begin{align*} 
{\mathcal S}_{t} & \leq 2^p \int_X \frac 1{\mu(B(x,t))} \int_{B(x,t)} | u_{B(x,t)} - u(y)|^p \, d \mu(y) \, d \mu (x) \\
& \leq c_P 2^p t^p \int_X  \frac{\nu (B(x,t))}{\mu(B(x,t))} \, d \mu(x) = c_P (2t)^p  \int_{X\times X} \frac{\chi_{\{d(x,y) \leq t\}}(x,y)}{\mu(B(x,t))} \, d \mu \otimes \nu (x,y).
\end{align*}
Notice now that if $d(x,y) \leq t$ then we have $B(y,t) \subseteq B(x,2t)$ and in particular, using the doubling condition, $\mu(B(y,t)) \leq \mu(B(x,2t)) \leq c_D \mu( B(x,t))$. Then we deduce that
$$ \frac{\chi_{\{d(x,y) \leq t\}}(x,y)}{\mu(B(x,t))} \leq c_D \frac{\chi_{\{d(x,y) \leq t\}}(x,y)}{\mu(B(y,t))}.$$
Using Fubini-Tonelli we then get ${\mathcal S}_{t} \leq c_Pc_D (2t)^p \nu (X)$.
\item[(iv)] In this case, we want to control the part where $d \geq 1$ and so we will use the triangular inequality $|u(x)-u(y)|^p \leq 2^{p-1} (|u(x)|^p + |u(y)|^p)$ and also that $\rho(x,y) \geq C \mu(B(x,d(x,y)))$ and $\rho(x,y) \geq C \mu(B(y,d(x,y)))$, to get
\begin{align*}
{\mathcal K}_t &=  \iint_{\{d(x,y) \leq t\}} \frac{|u(x)-u(y)|^p}{\rho(x,y)} \,d\mu(x) \, d \mu(y), \\
 &= \mathcal{K}_1 + \iint_{\{1 \leq d(x,y) \leq t\}} \frac{|u(x)-u(y)|^p}{\rho(x,y)} \,d\mu(x) \, d \mu(y) \\
 & \leq \mathcal{K}_1 +  \frac{2^p}C \iint_{\{1 \leq d(x,y) \leq t\}} \frac{|u(x)|^p}{\mu (B(x,d(x,y))) } \,d\mu(x) \, d \mu(y) \\
 & =  \mathcal{K}_1 +  \frac{2^p}C \int_X |u(x)|^p \int_{\{1 \leq d(x,y) \leq t\}} \frac{1}{\mu (B(x,d(x,y))) }  \, d \mu(y) \, d \mu(x).
 \end{align*}
In order to estimate the last integral we divide in shells $S_k=\{y\, : \,  2^k \leq d(x,y) \leq 2^{k+1} \}$ and then we have
\begin{align*}  \int_{\{1 \leq d(x,y) \leq t\}} \frac{1}{\mu (B(x,d(x,y))) }  \, d \mu(y)  & \leq \sum_{k=0}^ {\lfloor \log_2(t) \rfloor} \int_{S_k} \frac 1{\mu (B(x,2^k))} \, d \mu(y) \\ 
&=  \sum_{k=0}^ {\lfloor \log_2(t) \rfloor} \frac {\mu(S_k)}{\mu (B(x,2^k))} \leq\lfloor \log_2(2t) \rfloor \cdot (c_D-1),
\end{align*}
which concludes the proof.
\end{itemize}

%Now we perform the last integration and we use that $\mu (B(x,2t)) \leq c_A 2^Nt^N $
%$$ 
%\int_0^{1} \frac {{\mathcal S}_{2t}}{t^{N+ps+1}}dt \leq c_Pc_A 2^N4^p \nu(X) \int_0^{1} \frac {1} {t^{p(s-1)+1}}dt = \frac{c_Pc_A 2^N4^p }{p(1-s)} \nu(X).  
%$$
%Next we can use that ${\mathcal H}_t \leq c_A t^N 2^{p} \int u^p \, d \mu$ in order to handle the integral over $(1,\infty)$. 
%Therefore, by collecting the previous inequalities, we obtain
%$$ 
%\limsup_{ s \nearrow 1}  (1-s) \int_X \int_X \frac{ |u(x) - u(y)|^p}{d(x,y)^{N+ps}}  \, d \mu(x) \, d \mu(y)  \leq \frac{c_Pc_Ac_D 2^N4^p (N+p)}{p} \nu(X),
%$$
%which concludes the justification of the upper bound.

\end{proof}

\subsection{Lower bound of Theorem \ref{main}}

We first recall the following

\begin{definition}[Upper gradient]\rm 
	Given a function $f \in L^1+L^{\infty}$ and a function $g \geq 0$, we say that $g$ is an upper gradient up to scale $\delta$ of $f$ if for every curve $\gamma$ of length $\geq \delta$ we have 
	\begin{equation}\label{eqn:ug}
	|f(\gamma_a)-f(\gamma_b)| \leq \int_{\gamma} g.
	\end{equation}
\end{definition}
\noindent
Let us define $u^t =M_t u$ and
$$g_t(x') :=\frac 1{ \mu(B(x',t))^2}\iint_{B(x',t) \times B(x',t)} \left|\frac{u(x)-u(y)}t\right|^p \, d \mu(x) \, d \mu (y), $$
In this way we have 
$$
{\mathcal S}_t \geq t^{p} \int g_t(x')^p \, d \mu (x'),
$$ 
Now, the idea is that for some $C>0$, we have that $C g_{2t}$ is an upper gradient up to scale $t/2$ of the function $u^t$. This is significant thanks to Lemma \ref{lem:delta0} and Lemma \ref{lem:mtconv}.

The proof that $Cg_{2t}$ is  an upper gradient up to scale $t/2$ of $u^t$ is as follows: it is sufficient to check Equation \eqref{eqn:ug} only on curves that have length between $t/2$ and $t$, and then use the triangular inequality. So let us consider $\gamma: [a,b] \to X$ with length between $t/2$ and $t$. Then for every $c \in (a,b)$ we have $d(\gamma_c , \gamma_a) \leq t$ and $d (\gamma_c, \gamma_b) \leq t$. In particular $B(\gamma_a,t) \subseteq B(\gamma_c, 2t) \subseteq B(\gamma_a, 4t)$ and so

\begin{align*}
|u^t(\gamma_a) - u^t(\gamma_b)| & \leq \frac 1{\mu(B(\gamma_a,t)) \mu (B(\gamma_b,t))}  \int_{B(\gamma_a,t) \times B(\gamma_b,t)} | u(x)- u(y)|  \,d \mu(x) \, d \mu(y) \\
& \leq \frac 1{\mu(B(\gamma_a,t)) \mu (B(\gamma_b,t))} \int_{B(\gamma_c,2t) \times B(\gamma_c,2t)} | u(x) - u(y)| \,d \mu(x) \, d \mu(y) \\
& \leq \frac {c_D^4}{\mu(B(\gamma_a,4t)) \mu (B(\gamma_b,4t))} \int_{B(\gamma_c,2t) \times B(\gamma_c,2t)} | u(x) - u(y)| \,d \mu(x) \, d \mu(y) \\
& \leq 2tc_D^4 \cdot g_{2t} (\gamma_c,t) 
\end{align*}

In particular we have that, taking $h_t = 4c_D^4 g_{2t}$,
$$ \int_{\gamma} h_t  \geq \int_{\gamma} \frac 2{t} |u^t(\gamma_a) - u^t(\gamma_b)|  = \frac{2l(\gamma)}t  |u^t(\gamma_a) - u^t(\gamma_b)| \geq |u^t(\gamma_a) - u^t(\gamma_b)|.$$

Using Lemma \ref{lem:delta0} and \ref{lem:mtconv} we get $\liminf_{t \to 0} \frac{ \mathcal{S}_t}{t^p} \geq  C \cdot \Ch_p(X)$. Then we are done using $\mathcal{K}_t \geq c\mathcal{S}_{t/2}$

\begin{align*}
 \liminf_{s \to 1} (1-s)ps \int_0^{\infty} \frac{ \mathcal{K}_t }{t^{ps+1}} \, dt & \geq   \liminf_{s \to 1} (1-s)ps \int_0^1 \frac{ \mathcal{K}_t}{t^{ps+1}} \, dt \geq \liminf_{s \to 1} c \int_0^1  \frac{ \mathcal{S}_{t/2}}{t^p} \, d \nu_s \\ & \geq c \liminf_{t \to 0} \frac{ \mathcal{S}_{t/2}}{t^p} \geq  C \cdot \Ch_p(X),
 \end{align*}
 
where we used that $\nu_s=\frac{ (1-s)p }{t^{p(s-1)+1}}$ is a probability measure on $[0,1]$ that goes weakly to $\delta_0$.
%$$\nu (X) \leq \liminf_{t \to 0} \int_X h_t^p \, d \mu \leq \liminf_{s \to 1} \int_0^1 \int_X h_t^p \, d \mu   \,d \eta_s,$$
%where $\eta_s \to \delta_0$ as $s \to 1$.
% 
\section{Proof of Theorem \ref{main2}}

\subsection{Upper bound of Theorem \ref{main2}}
\noindent
%Let $\mu$ be a finite measure. 
%In this case we will use several times Lemma \ref{lem:key}. 
Consider the quantities
\begin{align*}
A_{\delta} &:= \mathop{\int_X\int_X}_{\{|u(x)-u(y)|>\delta\}} \frac{\delta^p}{\rho(x,y)d(x,y)^{p}}  \, d \mu(x) \, d \mu(y),  \\
B_{\delta, r} &:= \mathop{\int_X\int_X}_{\{|u(x)-u(y)|>\delta,\, d(x,y) \leq r \}} \frac{\delta^p}{\rho(x,y)d(x,y)^{p}}  \, d \mu(x) \, d \mu(y).
\end{align*}
Notice now that if $|u(x)-u(y)|>\delta$  we have $\delta^p \leq 2^{p-1} ( \delta^p \wedge |u(x)|^p +  \delta^p \wedge |u(y)|^p)$. Using this inequality and Lemma \ref{lem:key} we get
$$
B_{\delta, r} \leq A_{\delta} \leq  \frac{C}{r^{p}}\int_{X}(|u(x)|\wedge \delta)^p \, d \mu  + B_{\delta, r}. 
$$
In particular, thanks to dominated convergence, we deduce that for every $r>0$ the limit points as $\delta \to 0$ of $B_{\delta,r}$ and $A_{\delta}$ are the same. 
Now let us assume that $u \in W^{1,p} (X)$. In particular, by Proposition \ref{prop:HK} there exists a function $g \in L^p$ such that $\int g^p\, d \mu \leq C \cdot  \Ch_p (u)$ and for any $x,y$ with $d(x,y) \leq r$ we have
$$ 
|u(x)-u(y)| \leq d(x,y) ( g (x)+  g (y) ).
$$
But then, the triangular inequality let us conclude that in the subset $\{d(x,y) \leq r\}$ we have:
$$ 
\{  |u(x)-u(y)|>\delta \} \subseteq \{  Cd(x,y) \cdot g (x)  \geq \delta/2 \} \cup  \{ Cd(x,y) \cdot g (y)  \geq \delta/2 \}. 
$$
By symmetry then we can estimate
\begin{align*} B_{\delta,r}  & \leq 2 \mathop{\int_X\int_X}_{\{Cd(x,y) \cdot  g (x)  \geq \delta/2 , d(x,y) \leq r \}} \frac{\delta^p}{\rho(x,y) d(x,y)^{p}}  \, d \mu(x) \, d \mu(y) \\
& = 2 \int_X \int_{ d(x,y) \geq r_1(x) }  \frac{\delta^p}{\rho(x,y) d(x,y)^{p}}  \, d \mu(y) \, d \mu(x),
\end{align*}
where $r_1(x) = \frac { \delta}{2C g (x)}$. Using again Lemma \ref{lem:key} we get
$$ B_{\delta,r} \leq  \tilde C \int g(x)^p \, d \mu (x) \leq \tilde C \cdot C \cdot \Ch_p (u).$$ 

\noindent
\subsection{Lower bound of Theorem \ref{main2}} In this case we suppose, without loss of generality, that 
$$ \sup_{0 < \delta < 1} A_{\delta} \leq C.$$
Notice that then for every $r \leq 1$ we have
%- Second inequality (precisely as in Nguyen):
$$
C \geq \int_0^r \eps \delta^{\eps-1} A_{\delta} \, d \delta  = \frac{\eps}{p+\eps}  \int_X \int_X  \frac{\inf\{|u(x) - u(y)| , r\}^{p+\eps}}{\rho(x,y) d(x,y)^p} \,d \mu(x) \, d \mu (y). 
$$

Now let us define

$$ g_{t}: =\frac 1{ \mu(B(x',t))^2}\iint_{B(x',t) \times B(x',t)} \frac{\inf\{|u(x)-u(y))|,r\}}t \, d \mu(x) \, d \mu (y),$$
$$ g_{\varphi,t} :=\frac 1{ \mu(B(x',t))^2}\iint_{B(x',t) \times B(x',t)} \frac{|\varphi(u(x))-\varphi(u(y))|}t \, d \mu(x) \, d \mu (y),$$

and, with the same argument as in the proof of the lower bound in Theorem \ref{main}, we can estimate 

$$ C \geq \frac{\eps}{p+\eps}  \int_X \int_X  \frac{\inf\{|u(x) - u(y)| , r\}^{p+\eps}}{\rho(x,y) d(x,y)^p} \,d \mu(x) \, d \mu (y)  \geq  c \int_0^1 \frac{ \int_X g_t^{p+\eps} \, d \mu }{t^{p+\eps}} \, d \nu_{\eps}. $$

In particular there exists a sequence $t_{\eps} \to 0$ such that 
$$
\lim_{\eps \to 0} \int_X \left(\frac{g_{t_{\eps}}}{t_{\eps}} \right)^{p+\eps} \, d \mu  \leq C/c. 
$$
In particular, up to a subsequence we have $g_{t_{\eps}}/t_{\eps} \rightharpoonup h$ in $L^p_{{\rm loc}}(X, \mu)$ and $\int_X h^p \, d \mu  \leq C/c$.

Let us consider the class $\mathcal{L}_r \subset {\rm Lip}(\mathbb{R})$ of $1$-Lipschitz functions that have values in $[0,r]$; notice that for $\phi \in \mathcal{L}_r$ we have $|\phi(t)-\phi(s)|\leq |t-s|$ and $|\phi(t)-\phi(s)| \leq r$.  In particular we have $g_{\varphi, t} \leq g_t$; moreover we already know that, up to constants,  $g_{\varphi,t}$ is a weak upper gradient at scale $2t$ for $M_t (\varphi \circ u )$. This implies that $g_t$ is also a weak upper gradient at scale $2t$ for $M_t (\varphi \circ u )$ and using Lemma \ref{lem:delta0} and \ref{lem:mtconv} we find that $h$ is a $p$-weak upper gradient for $\varphi \circ u$ for every $\varphi \in \mathcal{L}_r$. 
Now we want to prove that $h$ is a $p$-weak upper gradient also for $u$. %In order to do this we use the technical lemma \ref{lem:tech}. 
Thanks to Definition \ref{def:wug} , we have that for every $\phi$ and every $p$-plan $\pi$ , there exists a set $\mathcal{N}_\phi$ that is $\pi$ negligible, such that for $\gamma \not \in \mathcal{N}_{\phi}$ we have we have $\phi \circ (f \circ \gamma) \in W^{1,1}(0,1)$ and $ |(\phi \circ (u \circ \gamma ))' (t) | \leq h  \circ \gamma(t)  | \gamma'|(t)$.
In particular, we can take a countable dense set $\mathcal{S} \subset \mathcal{L}_r$ and, denoting by 
$$
\mathcal{N}=\bigcup_{\phi \in \mathcal{S} }\mathcal{N}_{\phi},  
$$
if $\gamma \not \in \mathcal{N}$ we have that  $f= u \circ \gamma$ and $g=h\circ \gamma | \gamma'|$ satisfy the hypothesis of Lemma \ref{lem:tech}, and in particular we have
$$ \forall \gamma \not \in \mathcal{N} \; \; u \circ \gamma \in W^{1,1}(0,1) \qquad \text{ and} \qquad  |u \circ \gamma ' (t)| \leq g \circ \gamma (t) | \gamma' | \quad \text{ for a.e. }t \in [0,1]. $$
Since $\mathcal{N}$ is a union of countably many $\pi$-negligible sets, it is itself $\pi$-negligible. Thanks to the arbitrariness of $\pi$, using again Definition \ref{def:wug} we conclude that $h$ is indeed a $p$-weak upper gradient for $u$.

\begin{lemma}\label{lem:tech} Let us consider $f : [0,1] \to \mathbb{R}$. Suppose there exists $g \in L^1(0,1)$ such that for every $\phi$ belonging to a dense subset of $\mathcal{L}_r$ we have $\phi \circ f \in W^{1,1}(0,1)$ and 
$ |(\phi \circ f )' (t) | \leq g (t)$ for $\mathcal{L}$-almost every $t \in [0,1]$.
Then $f \in W^{1,1}(0,1)$ and $|f'| \leq g$.
\end{lemma}

\begin{proof} First of all let us observe that if the hypotesis is true for a dense subset of $\phi$ then it is true for every $\phi \in \mathcal{L}_r$ since it is equivalent to require 
$$|\phi(f(x))-\phi(f(y))| \leq \int_x^y  g(t) \, d t  \qquad \text{ for almost every }x<y \in [0,1],$$
which is a condition stable for uniform convergence of $\phi$.

Let us consider, for every $n \in \mathbb{N}$
$$\phi_n (t)=\begin{cases} 0 \qquad& \text{ if }t < rn \\ t-rn & \text{ if } rn \leq t < r(n+1) \\ r  & \text{ if }t \geq r(n+1); \end{cases}$$
we also define $\phi_{-n}(t)=-\phi_n(-t)$. Then clearly we have $\phi_n \in \mathcal{L}_r$; moreover $$\sum_{n \in \mathbb{Z}} \phi_n(t)=t.$$ 
Considering then $f_n= \phi_n \circ f$ we have 
\begin{equation}\label{eqn:sumf} \sum_{n \in \mathbb{Z}} f_n (x) =f(x) \qquad \text{ for every $x \in [0,1]$.}\end{equation}
 By hypothesis we have $f_n \in W^{1,1}$ and $|f_n'| \leq g$; however we have $f_n'=0$ almost everywhere in $\{f_n=0\} \cup \{f_n=r\}$ thnaks to stardand Sobolev theory. In particular denoting with $A_n= \{ rn \leq f< r(n+1)\}$ we have more precisely $|f_n'| \leq g \chi_{A_n}$.
Let us consider $N$ big enough such that $\{f \leq Nr\}$ is not negligible. Then we have that $\{f_n=0\}$ is not negligible for $n \geq N$ and then we have $\|f_n\|_{\infty} \leq \| g \chi_{A_n}\|_1$ thanks to the fact that there exists $x_0$ such that $f_n(x_0)=0$ and the estimate 
$$|f_n(x)|=|f_n(x)-f_n(x_0)| =\left| \int_{x_0}^x f_n' (y) \, dy \right| \leq \int_0^1 |f_n'|(y) \, d y \leq \| g \chi_{A_n}\|_1.$$ 
A similar argument can be used for $n$ very negative. Now we have that  $\| g \chi_{A_n} \|_1$ is summable and adds up to $\|g \|_1$. In particular this proves that $\sum_{|n| \leq N} f_n$ converges in $L^{\infty}$ to some function $\tilde f$ which will coincide with $f$ almost everywhere thanks to \eqref{eqn:sumf}. We will in particular have that 
$$ \sum_{|n| \leq N} f_n \stackrel{L^1}{\longrightarrow} f \qquad \sum_{|n| \leq N} f_n' \stackrel{L^1}{\longrightarrow} 	\bar{g} $$
where $\bar(g) = f_n' $ in $A_n$. In particular we have $f \in W^{1,1}$ and $f'=\bar{g}$; in particular $|f'| =|\bar{g}| \leq g$.
\end{proof}

\bigskip

\end{document}